\documentclass{amsart}
\usepackage{amsmath}
\usepackage{amsfonts}
\usepackage{graphics}
\usepackage{epsfig}
\usepackage{amssymb}
\usepackage{amscd}
\usepackage[all]{xy}
\usepackage{latexsym}
\usepackage{graphicx}
\usepackage{multirow}
\usepackage{geometry}
\usepackage{geometry}

\usepackage{color}

\theoremstyle{Theorem}
\newtheorem{theorem}{Theorem}[section]
\newtheorem{defn}[theorem]{Definition}
\newtheorem{thm}[theorem]{Theorem}

\theoremstyle{remark}
\newtheorem{lem}{\bf Lemma}[section]
\newtheorem{cor}{\bf Corollary}[section]

\newtheorem{prop}{\bf Proposition}[section]

\numberwithin{equation}{section} \numberwithin{figure}{section}

\renewcommand*{\to}{\rightarrow}
\renewcommand*{\bar}[1]{\overline{#1}}

\newcommand{\mb}[1]{\mathbb{#1}} 

\newcommand{\mc}[1]{\mathcal{#1}}

\newcommand{\mf}[1]{\mathbf{#1}}

\title{Asymptotic Expansion of the Gaussian Integral Operators on Riemannian submanifolds of $\mb R^{n}$}

\author{Jia-Ming (Frank) Liou}
\address{Department of Mathematics\\
National Cheng Kung University\\
No.1, University Road, Tainan City 701, Taiwan\\ fjmliou@mail.ncku.edu.tw}

\author{Chi-Chien Lu}
\address{Department of Mathematics\\
National Cheng Kung University\\
No.1, University Road, Tainan City 701, Taiwan\\ l18121030@gs.ncku.edu.tw}

\begin{document}\fontsize{12}{20pt}\selectfont
\maketitle

\begin{abstract}
The Gaussian integral operator arises naturally as a local Euclidean approximation of the heat semigroup on a Riemannian manifold and plays a pivotal role in the analysis of graph Laplacians, particularly within the frameworks of manifold learning and spectral graph theory. In this paper, we study the asymptotic behavior of the Gaussian integral operator on a smooth Riemannian submanifold \( M \subset \mathbb{R}^n \), focusing on its expansion as \( \varepsilon \to 0^+ \). Under the assumption that the input function is real analytic near a fixed point \( x \in M \), we derive a full asymptotic expansion of the operator and compute the first-order correction term explicitly in terms of the mean curvature vector and the scalar curvature of the submanifold. In particular, we apply our results to hypersurfaces in Euclidean space and investigate geometric conditions under which points exhibit \emph{equicurvature}.
\end{abstract}


\section{Introduction}
The Gaussian integral operator plays a crucial role in the analysis of graph Laplacians, particularly in the fields of manifold learning and spectral graph theory. By examining the spectral properties of the graph Laplacian, one can effectively perform dimensionality reduction by embedding data into a lower-dimensional Euclidean space. For example, diffusion maps utilize the graph Laplacian to uncover intrinsic low-dimensional structures within high-dimensional datasets. For related works, see \cite{BN}, \cite{CS}, and \cite{CLL}.

In the context of manifold learning, the normalized graph Laplacian serves as a discrete analogue of the Laplace–Beltrami operator on a Riemannian manifold. A crucial step in establishing the convergence of normalized graph Laplacians to their continuous counterparts involves a detailed analysis of the Gaussian integral operator. This investigation is foundational to both the theoretical understanding and the practical implementation of these operators in analyzing the geometric structure of high-dimensional data. For studies on the convergence of the normalized graph Laplacian when data points lie on a submanifold \( M \subset \mathbb{R}^n \), see, for example, \cite{DB}, \cite{S1}, \cite{S2}, \cite{S3}, \cite{HAL}, \cite{BN1}, \cite{BN3}, \cite{SA}, \cite{CS}, etc.

Before introducing the definition of the Gaussian integral operator, we briefly recall its motivation in a broader analytical context. The origin of this operator lies in the theory of heat operators. The \emph{heat operator} \( \mathcal{H}_t = e^{t\Delta} \) on a \( d \)-dimensional compact, oriented Riemannian manifold\footnote{Throughout this paper, all manifolds are assumed to be connected.} \( M \), is defined as the one-parameter semigroup generated by the Laplace–Beltrami operator \( \Delta \). It can be expressed as a family of integral operators parameterized by time \( t > 0 \):
\[
(\mathcal{H}_t f)(x) = \int_M H(x, y, t) f(y) \, dV(y),
\]
where \( dV \) denotes the Riemannian volume form on \( M \), and \( H : M \times M \times [0, \infty) \to \mathbb{R} \) is the \emph{heat kernel}, which satisfies the heat equation
\[
\frac{\partial H}{\partial t} = \Delta_x H.
\]
When both the geodesic distance \( d(x, y) \) and the time \( t \) are sufficiently small, the heat kernel admits the following classical asymptotic expansion:
\[
H(x, y, t) = \frac{1}{(4\pi t)^{d/2}} \exp\left(-\frac{d(x, y)^2}{4t}\right) \left( u_0(x, y) + t u_1(x, y) + O(t^2) \right),
\]
where \( u_0(x, y) = 1 + O(d(x, y)^2) \).

In practical applications, particularly when \( M \) is viewed as a Riemannian submanifold of \( \mathbb{R}^n \), it is often more convenient to use the ambient Euclidean distance \( \|x - y\|_{\mathbb{R}^n} \) instead of the intrinsic geodesic distance \( d(x, y) \). This leads naturally to the formulation of the \emph{Gaussian integral operator}, constructed using a Gaussian kernel based on Euclidean distance. Specifically, the \emph{Gaussian kernel} \( k_{\varepsilon} : M \times M \to \mathbb{R} \) is defined by
\[
k_{\varepsilon}(x, y) = \frac{1}{(4\pi\varepsilon)^{d/2}} \exp\left(-\frac{\|y - x\|_{\mathbb{R}^n}^2}{4\varepsilon}\right),
\]
where \( \varepsilon > 0 \) is a small positive parameter.

For each \( \varepsilon > 0 \), we define the associated \emph{Gaussian integral operator}
\[
\mathcal{K}_\varepsilon : C^\infty(M) \to C^\infty(M),
\]
by
\[
(\mathcal{K}_\varepsilon f)(x) = \int_M k_\varepsilon(x, y) f(y) \, dV(y),
\]
where \( dV \) is the Riemannian volume form on \( M \).

Our investigation into the asymptotic behavior of \( \mathcal{K}_\varepsilon \) is motivated by its role in the convergence analysis of the normalized graph Laplacian to the Laplace–Beltrami operator on Riemannian submanifolds of Euclidean space. According to \cite{S2}, it was claimed that in unpublished notes, G.~Tokarev was the first to derive the leading two terms in the asymptotic expansion of the Gaussian integral operator. The works concerning the convergence of graph Laplacians can be viewed as early steps toward a systematic understanding of the asymptotic behavior of Gaussian integral operators.

In this paper, we aim to derive the \emph{asymptotic expansion} of \( (\mathcal{K}_\varepsilon f)(x) \) as \( \varepsilon \to 0^+ \), assuming that \( f \) is analytic in a neighborhood of \( x \in M \). Specifically, we seek a sequence of smooth functions \( \{a_n : M \to \mathbb{R}\}_{n=0}^{\infty} \) such that
\[
(\mathcal{K}_\varepsilon f)(x) \sim \sum_{n=0}^{\infty} a_n(x) \, \varepsilon^n,
\]
meaning that for each \( N \in \mathbb{N} \), the remainder satisfies
\[
(\mathcal{K}_\varepsilon f)(x) - \sum_{n=0}^{N} a_n(x) \, \varepsilon^n = o(\varepsilon^N) \quad \text{as } \varepsilon \to 0^+.
\]
The coefficients \( a_n(x) \) can be defined inductively by the recurrence
\[
a_{n+1}(x) := \lim_{\varepsilon \to 0^+} \frac{(\mathcal{K}_\varepsilon f)(x) - \sum_{k=0}^{n} a_k(x) \, \varepsilon^k}{\varepsilon^{n+1}},
\]
provided that the limit exists.

In Section~2, we analyze the relationship between Euclidean and geodesic distances for points on a submanifold of Euclidean space and revisit the Taylor expansion of the Riemannian volume form, following the work of A.~Gray~\cite{AG}.  
In Section~3, we compute the asymptotic expansion of the Gaussian integral operator and express the coefficient \( a_1 \) explicitly in terms of the mean curvature vector and scalar curvature.  
In Section~4, we apply the Gaussian integral operator to hypersurfaces in Euclidean space to investigate conditions under which points on a hypersurface exhibit \emph{equicurvature}. In particular, we classify the equicurved condition in the case where the hypersurface is a surface in \( \mathbb{R}^3 \). To the best of our knowledge, no explicit examples of compact equicurved hypersurfaces are known when \( d \geq 3 \). This observation suggests a natural and intriguing open problem in differential geometry.

\section{Geometric Expansions in Normal Coordinates}
In this section, we present the Taylor expansions of two essential quantities in normal coordinates required for computing the asymptotic expansion of the Gaussian integral operator. Specifically, we derive the Taylor expansion of the squared Euclidean distance function composed with the exponential map, as well as the expansion of the Riemannian volume form expressed in geodesic normal coordinates. We also obtain the corresponding expansions in geodesic polar coordinates, where the local geometry is characterized in terms of radial distance and directions on the unit sphere.

In \cite{S1} and \cite{S2}, O.~G. Smolyanov, H.~v.~Weizsäcker, and O.~Wittich prove the following result:

\begin{thm}[\cite{S1}, \cite{S2}]
Consider points \( x, y \in M \subset N \), where \( M \) is an isometrically embedded submanifold via the map \( \phi \). Let \( p \in U(x) \), where \( U(x) \subset L \) is sufficiently small so that \( x \) and \( y \) are joined by a unique minimizing geodesic \( \gamma_{xy}^{M} \) in \( M \) starting at \( x \). Assume \( \gamma_{xy}^{M} \) is parametrized by arc length. Then
\[
\lim_{d_{M}(x,y)\to 0}\frac{d_{N}(\phi(x),\phi(y))^{2}-d_{M}(x,y)^{2}}{d_{M}(x,y)^{4}} = -\frac{1}{12}\left\|\mathsf{B}_{\phi}(\dot{\gamma}_{xy}^{M},\dot{\gamma}_{xy}^{M})\right\|^{2},
\]
where \( \mathsf{B}_{\phi} \) denotes the second fundamental form of the embedding \( \phi : M \to N \).
\end{thm}

This identity is equivalent to the expansion:
\[
d_{N}(\phi(x),\phi(y))^{2} = d_{M}(x,y)^{2} - \frac{d_{M}(x,y)^{4}}{12} \left\|\mathsf{B}_{\phi}(\dot{\gamma}_{xy}^{M},\dot{\gamma}_{xy}^{M})\right\|^{2} + O(d_{M}(x,y)^{5}).
\]

Their proof uses the Taylor expansion of the Riemannian metric in geodesic normal coordinates to expand \( d_{N}(\cdot, \phi(x))^2 \) near \( x \). In this section, we provide a much simpler proof of this identity in the special case where the ambient manifold is \( \mathbb{R}^n \). To this end, we begin by reviewing some basic facts about the Taylor expansion of scalar functions. 

Recall that a function \( h : U \subset \mathbb{R}^d \to \mathbb{R} \), defined on an open neighborhood \( U \) of the origin, is said to be \emph{analytic at the origin} if there exists \( \delta > 0 \) such that its Taylor series expansion
\begin{equation}\label{t1}
h(\mathbf{s}) = \sum_{j=0}^\infty \frac{1}{j!} H_j(h)(\mathbf{s})
\end{equation}
converges on the open ball \( B_0(\delta) \subset \mathbb{R}^d \). Here, \( \mathbf{s} = (s_1, \dots, s_d) \) denotes the standard coordinates on \( \mathbb{R}^d \), and each term \( H_j(h) \) is given by
\[
H_j(h)(\mathbf{s}) = \sum_{|\alpha| = j} \binom{j}{\alpha} (D^\alpha h)(0) \, \mathbf{s}^\alpha,
\]
where \( \alpha = (\alpha_1, \dots, \alpha_d) \in \mathbb{Z}_{+}^d \) is a multi-index of total degree \( |\alpha| = \alpha_1 + \cdots + \alpha_d \), and
\[
\binom{j}{\alpha} = \frac{j!}{\alpha_1! \cdots \alpha_d!}, \quad
\mathbf{s}^\alpha = s_1^{\alpha_1} \cdots s_d^{\alpha_d}, \quad
D^\alpha h(0) = \left. \frac{\partial^{|\alpha|} h}{\partial s_1^{\alpha_1} \cdots \partial s_d^{\alpha_d}} \right|_{\mathbf{s} = 0}.
\]
Each \( H_j(h) \) is either zero or a homogeneous polynomial of degree \( j \). For notational convenience, we write \( H_j(h) = h_j \).

Now let \( \mathbf{s} = s \mathbf{v} \), where \( 0 < s < \delta \) and \( \mathbf{v} \in S^{d-1} \). Then the analyticity of \( h \), together with the homogeneity of each \( h_j \), implies that
\begin{equation}\label{t2}
h(s \mathbf{v}) = \sum_{j=0}^{\infty} \frac{s^j}{j!} h_j(\mathbf{v}).
\end{equation}
Consequently, the function \( h \), when restricted to the punctured ball \( B_0'(\delta) := B_0(\delta) \setminus \{0\} \), admits the expansion \eqref{t2} in polar coordinates \( (s, \mathbf{v}) \).

Let \( M \) be a Riemannian submanifold of dimension \( d \) in \( \mathbb{R}^n \), and let \( x \in M \) be a point. Denote by \( \operatorname{inj}_x(M) \) the injectivity radius of \( M \) at \( x \).
Choose an ordered orthonormal basis \( \{ (e_1)_x, \dots, (e_d)_x \} \) of the tangent space \( T_x M \). For any \( \delta > 0 \) such that \( \delta < \operatorname{inj}_x(M) \), define the map
\(
\mathbf{x} : B_0(\delta) \to M
\)
by
\[
\mathbf{x}(s_1, \dots, s_d) = \exp_x\left( \sum_{i=1}^d s_i (e_i)_x \right),
\]
where \( B_0(\delta) \subset \mathbb{R}^d \) denotes the open Euclidean ball of radius \( \delta \) centered at the origin.
The map \( \mathbf{x} \) is a diffeomorphism from \( B_0(\delta) \) onto its image \( \mathcal{B}_x(\delta) \subset M \), which is an open neighborhood of \( x \). We refer to the map \( \mathbf{x} \), or equivalently to the coordinate functions \( (s_1, \dots, s_d) \), as a system of \emph{geodesic normal coordinates} centered at \( x \). The image \( \mathcal{B}_x(\delta) \) is called the \emph{geodesic ball} of radius \( \delta \) centered at \( x \).
Note that the system of geodesic normal coordinates depends on the choice of the ordered orthonormal basis of \( T_x M \).

Let \( g : B_0(\delta) \to \mathbb{R} \) be the function defined by
\[
g(\mathbf{s}) = \| \mathbf{x}(\mathbf{s}) - x \|_{\mathbb{R}^n}^2,
\]
where \( \mathbf{x} : B_0(\delta) \to M \subset \mathbb{R}^n \) is the exponential coordinate chart centered at \( x \in M \). The function \( g \) is evidently smooth, and we further assume it is analytic on the open ball \( B_0(\delta) \). In what follows, we briefly outline the computation of the Taylor expansion of \( g \) about the origin. It is immediate that \( g(0) = 0 \), and the first-order partial derivatives are given by
\[
\frac{\partial g}{\partial s_i} = 2 \left\langle \frac{\partial \mathbf{x}}{\partial s_i}, \mathbf{x}(\mathbf{s}) - x \right\rangle,
\]
so in particular \( \frac{\partial g}{\partial s_i}(0) = 0 \) for all \( i \). Applying the product rule, we compute the second-order derivatives:
\begin{align*}
\frac{\partial^2 g}{\partial s_i \partial s_j} 
&= 2 \left\langle \frac{\partial^2 \mathbf{x}}{\partial s_i \partial s_j}, \mathbf{x}(\mathbf{s}) - x \right\rangle 
+ 2 \left\langle \frac{\partial \mathbf{x}}{\partial s_i}, \frac{\partial \mathbf{x}}{\partial s_j} \right\rangle \\
&= 2 \left\langle \frac{\partial^2 \mathbf{x}}{\partial s_i \partial s_j}, \mathbf{x}(\mathbf{s}) - x \right\rangle 
+ 2 g_{ij}(\mathbf{s}).
\end{align*}
Evaluating at \( \mathbf{s} = 0 \), we find
\[
\frac{\partial^2 g}{\partial s_i \partial s_j}(0) = 2 g_{ij}(0) = 2 \delta_{ij}.
\]
Furthermore, since the Christoffel symbols satisfy \( \Gamma_{ij}^{k}(x) = 0 \) at the center of normal coordinates, it follows that
\[
\frac{\partial^3 \mathbf{x}}{\partial s_i \partial s_j \partial s_k}(0) = 0.
\]
More generally, higher-order derivatives \( D^\alpha g(0) \) can be computed by repeated application of the product rule, taking into account the symmetries of the exponential map and the vanishing of certain derivatives at the origin. Consequently, the Taylor expansion of \( g(\mathbf{s}) \) near the origin begins as
\[
g(\mathbf{s}) = s^2 + O(s^4),
\]
where \( s = \sqrt{s_1^2 + \cdots + s_d^2} \). This expression reflects the fact that
\[
\| y - x \|_{\mathbb{R}^n}^2 = d(x, y)^2 + O(d(x, y)^4),
\]
for any point \( y \in \mathcal{B}_x(\delta) \), where \( d(x, y) \) denotes the intrinsic (geodesic) distance between \( x \) and \( y \). The fourth-order term can be expressed in terms of the second fundamental form via the Gauss equation. Details can be found in \cite{S1} and \cite{S2}.

We now present an alternative approach to derive the geometric expansion of the squared Euclidean norm. By restricting the function \( g \) to the punctured ball \( B_0'(\delta) \), and evaluating its Taylor expansion along each ray of the form \( \mathbf{s} = s \mathbf{v} \), we obtain the corresponding expansion in polar coordinates \( (s, \mathbf{v}) \):
\[
g(\mathbf{s}) = \sum_{j=0}^{\infty} \frac{s^j}{j!} \, g_j(\mathbf{v}),
\]
where each \( g_j(\mathbf{v}) \) is the homogeneous degree-\( j \) component of the Taylor series in the direction \( \mathbf{v} \in S^{d-1} \). From previous computations, we have
\[
g_0(\mathbf{v}) = 0, \quad g_1(\mathbf{v}) = 0, \quad g_2(\mathbf{v}) = 2, \quad g_3(\mathbf{v}) = 0.
\]
To compute the coefficients \( g_j \) more directly, we use the series expansion of \( g \) in polar coordinates. Specifically, for each fixed direction \( \mathbf{v} \in S^{d-1} \), we define the one-variable function
\[
g_{\mathbf{v}} : (-\delta, \delta) \to \mathbb{R}, \quad g_{\mathbf{v}}(t) := g(t \mathbf{v}),
\]
which corresponds to the restriction of \( g \) along the radial line in the direction \( \mathbf{v} \). Then, the coefficient \( g_j(\mathbf{v}) \) in the polar expansion of \( g \) is given by the \( j \)-th derivative of \( g_{\mathbf{v}} \) evaluated at the origin:
\[
g_j(\mathbf{v}) = \left. \frac{d^j}{dt^j} g_{\mathbf{v}}(t) \right|_{t = 0}.
\]

Let \( \mathbf{v} = (v_1, \dots, v_d) \in S^{d-1} \) be a unit vector. With respect to the chosen orthonormal basis \( \{ (\mathbf{e}_i)_x \}_{i=1}^d \) of \( T_x M \), we associate the corresponding tangent vector \( \mathbf{v}_x \in T_x M \) by
\(
\mathbf{v}_x = \sum_{i=1}^d v_i (\mathbf{e}_i)_x.
\)
Let
\[
\gamma_{\mathbf{v}_x} : (-\delta, \delta) \to \mathcal{B}_x(\delta)
\]
denote the unique geodesic satisfying the initial conditions
\[
\gamma_{\mathbf{v}_x}(0) = x, \quad \gamma_{\mathbf{v}_x}'(0) = \mathbf{v}_x.
\]
Then, for all \( |t| < \delta \), we have
\[
\gamma_{\mathbf{v}_x}(t) = \exp_x(t \mathbf{v}_x).
\] 
It follows from the definition that
\[
g_{\mathbf{v}}(t) = \| \gamma_{\mathbf{v}_{x}}(t) - x \|_{\mathbb{R}^n}^2.
\]
We now proceed to compute the Taylor expansion of \( g_{\mathbf{v}}(t) \) about \( t = 0 \). To express the derivatives of \( g_{\mathbf{v}}(t) \) at \( t = 0 \), we require the notion of the second fundamental form of a submanifold \( M \subset \mathbb{R}^n \). We briefly review its definition below.

Let \( \bar{\nabla} \) be the Riemannian connection on \( \mathbb{R}^n \). For each point \( x \in M \), we have the orthogonal direct sum decomposition
\[
T_x \mathbb{R}^n = T_x M \oplus T_x M^{\perp}.
\]
We denote the orthogonal projections from \( T_x \mathbb{R}^n \) onto \( T_x M \) and \( T_x M^{\perp} \) by \( T_x \) and \( \perp_x \), respectively.

Let \( X \) and \( Y \) be vector fields defined on an open neighborhood \( U \subset M \) of a point \( x \in M \). Suppose \( \bar{X} \) and \( \bar{Y} \) are smooth vector fields defined on an open neighborhood \( \bar{U} \subset \mathbb{R}^n \) of \( x \), with \( \bar{U} \cap M = U \), such that \( \bar{X} = X \) and \( \bar{Y} = Y \) on \( U \). Then, for any \( y \in U \), the Levi-Civita connection \( \nabla \) on \( M \) satisfies
\[
(\nabla_X Y)(y) = T_y \left( \bar{\nabla}_{\bar{X}} \bar{Y} (y) \right).
\]
We define the second fundamental form
\[
\mathbf{B}_x :\; T_x M \times T_x M \to T_x M^{\perp}
\]
by
\(
\mathbf{B}_x(X(x), Y(x)) =(\bar{\nabla}_{\bar{X}(x)} \bar{Y})^{\perp_x}.
\)
The following result provides the foundation for the expansion of \( g_{\mathbf{v}} \).
\begin{lem}\label{tmetric}
Let \( \mathbf{v} \in S^{d-1} \) be a unit vector. The function \( g_{\mathbf{v}} \) defined above admits a Taylor expansion around \( t = 0 \) of the form
\[
g_{\mathbf{v}}(t) = \sum_{j=0}^{\infty} \frac{g_j(\mathbf{v})}{j!} t^j, \quad \text{for } |t| < \delta.
\]
Each coefficient \( g_j(\mathbf{v}) \) is either zero or a homogeneous polynomial of degree \( j \) in \( \mathbf{v} = (v_1, \dots, v_d) \). The first few coefficients are given by
\[
g_0(\mathbf{v}) = 0, \quad 
g_1(\mathbf{v}) = 0, \quad 
g_2(\mathbf{v}) = 2, \quad 
g_3(\mathbf{v}) = 0, \quad 
g_4(\mathbf{v}) = -2 \| \mathbf{B}_x(\mathbf{v}_x, \mathbf{v}_x) \|^2,
\]
where \( \mathbf{B}_x(\mathbf{v}_x, \mathbf{v}_x) \) denotes the second fundamental form of the submanifold \( M \subset \mathbb{R}^n \) at the point \( x \), evaluated on the vector \( \mathbf{v}_x \in T_x M \).
\end{lem}
\begin{proof}
By definition, \( g_k(\mathbf{v}) = g_{\mathbf{v}}^{(k)}(0) \), so it suffices to compute the derivatives \( g_{\mathbf{v}}^{(k)}(0) \). Recall that
\[
g_{\mathbf{v}}(t) = \langle \gamma_{\mathbf{v}_{x}}(t) - x, \gamma_{\mathbf{v}_{x}}(t) - x \rangle_{\mathbb{R}^n}.
\]
Define \( \xi_{\mathbf{v}_{x}}(t) := \gamma_{\mathbf{v}_{x}}(t) - x \). Then \( \xi_{\mathbf{v}_{x}}(0) = 0 \), and \( \xi_{\mathbf{v}_{x}}^{(k)}(t) = \gamma_{\mathbf{v}_{x}}^{(k)}(t) \) for all \( k \geq 1 \). Since
\[
g_{\mathbf{v}_{x}}(t) = \langle \xi_{\mathbf{v}_{x}}(t), \xi_{\mathbf{v}_{x}}(t) \rangle,
\]
we apply the Leibniz rule to obtain
\[
g_{\mathbf{v}}^{(k)}(0) = \sum_{j=0}^{k} \binom{k}{j} \langle \xi_{\mathbf{v}_{x}}^{(j)}(0), \xi_{\mathbf{v}_{x}}^{(k-j)}(0) \rangle.
\]

The first few derivatives are:
\begin{align*}
g_{\mathbf{v}}'(0) &= 2 \langle \xi_{\mathbf{v}_{x}}(0), \xi_{\mathbf{v}_{x}}'(0) \rangle = 2 \langle 0, \mathbf{v}_{x} \rangle = 0, \\
g_{\mathbf{v}}''(0) &= 2 \| \gamma_{\mathbf{v}_{x}}'(0) \|^2 = 2 \| \mathbf{v}_{x} \|^2 = 2, \\
g_{\mathbf{v}}^{(3)}(0) &= 6 \langle \gamma_{\mathbf{v}_{x}}^{(2)}(0), \gamma_{\mathbf{v}_{x}}'(0) \rangle = 6 \langle \gamma_{\mathbf{v}_{x}}^{(2)}(0), \mathbf{v}_{x} \rangle.
\end{align*}

Since \( \gamma_{\mathbf{v}_{x}} \) is a geodesic in \( M \), it satisfies
\[
\langle \gamma_{\mathbf{v}_{x}}^{(2)}(t), \gamma_{\mathbf{v}_{x}}'(t) \rangle = 0 \quad \text{for all } t,
\]
so in particular,
\[
g_{\mathbf{v}}^{(3)}(0) = 6 \langle \gamma_{\mathbf{v}_{x}}^{(2)}(0), \mathbf{v}_{x} \rangle = 0.
\]

Differentiating the identity above with respect to \( t \), we obtain
\[
\langle \gamma_{\mathbf{v}_{x}}^{(3)}(t), \gamma_{\mathbf{v}_{x}}'(t) \rangle + \| \gamma_{\mathbf{v}_{x}}^{(2)}(t) \|^2 = 0.
\]
Evaluating at \( t = 0 \), we find
\[
\langle \gamma_{\mathbf{v}_{x}}^{(3)}(0), \mathbf{v}_{x} \rangle = - \| \gamma_{\mathbf{v}_{x}}^{(2)}(0) \|^2.
\]

Now we compute the fourth derivative:
\begin{align*}
g_{\mathbf{v}}^{(4)}(0)
&= 8 \langle \gamma_{\mathbf{v}_{x}}^{(3)}(0), \gamma_{\mathbf{v}_{x}}'(0) \rangle + 6 \| \gamma_{\mathbf{v}_{x}}^{(2)}(0) \|^2 \\
&= 8 \cdot \left( - \| \gamma_{\mathbf{v}_{x}}^{(2)}(0) \|^2 \right) + 6 \| \gamma_{\mathbf{v}_{x}}^{(2)}(0) \|^2 \\
&= -2 \| \gamma_{\mathbf{v}_{x}}^{(2)}(0) \|^2.
\end{align*}

Finally, using the fact that \( \gamma_{\mathbf{v}_{x}}^{(2)}(0) =\mf B_x(\mathbf{v}_{x}, \mathbf{v}_{x}) \), we obtain
\[
g_{\mathbf{v}}^{(4)}(0) = -2 \|\mf B_x(\mathbf{v}_{x}, \mathbf{v}_{x}) \|^2,
\]
as claimed.
\end{proof}

Let \( \mathcal{B}_x'(\delta) \) represent the punctured geodesic ball centered at \( x \), defined as follows:
\[
\mathcal{B}_x'(\delta) = \mathcal{B}_x(\delta) \setminus \{x\}.
\]
Building on the previous result, we now present the following corollary, which provides the expansion we aimed to demonstrate:
\begin{cor}
For every \( y \in \mathcal{B}_x'(\delta) \), the squared Euclidean distance can be expressed as:
\[
\| y - x \|_{\mathbb{R}^n}^2 = s^2 - \frac{s^4}{12} \left\| \mathsf{B}_x(\mathbf{v}_x, \mathbf{v}_x) \right\|^2 + O(s^5),
\]
where \( s = d(x, y) \) and \( \mathbf{v}_x \in T_x M \) is the unique unit tangent vector corresponding to the point \( y = \exp_x(s \mathbf{v}_x) \).
\end{cor}

We now recall a result from \cite{AG} concerning the Taylor expansion of the volume form on a Riemannian manifold. To state this result precisely, we first introduce the necessary terminology.

Let \( \nabla \) denote the Levi-Civita connection on the Riemannian manifold \( M \). For smooth vector fields \( X \) and \( Y \), the Riemann curvature operator \( R_{XY} \) is defined by
\[
R_{XY} := \nabla_{[X, Y]} - [\nabla_X, \nabla_Y],
\]
where \( [X, Y] \) denotes the Lie bracket of \( X \) and \( Y \), and \( [\nabla_X, \nabla_Y] := \nabla_X \nabla_Y - \nabla_Y \nabla_X \) is the commutator of covariant derivatives.

Let \( Y_1, \ldots, Y_m \) be smooth vector fields on \( M \), and let \( Z \) denote a tensor field on \( M \). The \( m \)-th covariant derivative of \( Z \) along the vector fields \( Y_1, \ldots, Y_m \) is defined inductively by:
\[
\nabla_{Y_1 \cdots Y_m}^m Z = \nabla_{Y_m} \left( \nabla_{Y_1 \cdots Y_{m-1}}^{m-1} Z \right),
\]
with the base case given by \( \nabla_Y^1 Z = \nabla_Y Z \).

We now fix normal coordinates $\mf x:B_{0}(\delta)\to\mc B_{x}(\delta)$, as previously described, and we define the associated local frame fields:
\[
X_i = \frac{\partial}{\partial s_i}, \quad \text{for } 1 \leq i \leq d.
\]
The collection \( \{X_i\}_{i=1}^d \) forms a smooth local frame in the geodesic ball \( \mathcal{B}_x(\delta) \). For convenience, we adopt the following notation:
\[
\nabla_{i_1 \cdots i_k}^k = \nabla_{X_{i_1} \cdots X_{i_k}}^k, \qquad 
R_{ijkl} = \langle R_{X_i X_j} X_k, X_l \rangle, \qquad 
R_{ij} = \sum_{k=1}^d R_{ikjk},
\]
where \( 1 \leq i, j \leq d \) and \( 1 \leq i_1, \ldots, i_k \leq d \).

The pullback of the Riemannian volume form \( dV \) under the normal coordinate parametrization
\(
\mathbf{x} : B_0(\delta) \to \mathcal{B}_x(\delta)
\)
can be expressed as
\[
\mathbf{x}^* dV = \rho(\mathbf{s}) \, ds_1 \wedge \cdots \wedge ds_d,
\]
where \( \rho(\mathbf{s}) \) is a smooth positive function representing the volume density in normal coordinates. The following result is attributed to Alfred Gray:

\begin{thm}\label{g1}
In a neighborhood of \( 0 \in \mathbb{R}^d \), the volume density function \(\rho(s_1, \dots, s_d) \), defined by the pullback of the Riemannian volume form under normal coordinates centered at \( x \in M \), admits the Taylor expansion:
\[
\rho(\mathbf{s}) = \sum_{k=0}^{\infty} \frac{1}{k!} \rho_k(\mathbf{s}),
\]
where each \( \rho_k \) is either zero or a homogeneous polynomial of degree \( k \) in \( \mathbf{s} = (s_1, \dots, s_d) \). The first few coefficients are given by:
\[
\rho_0(\mathbf{s}) = 1, \quad 
\rho_1(\mathbf{s}) = 0, \quad 
\rho_2(\mathbf{s}) = - \frac{1}{3}\sum_{i,j=1}^{d} R_{ij}(x) s_i s_j,
\]
\[
\rho_3(\mathbf{s}) = -\frac{1}{2} \sum_{i,j,k=1}^{d} (\nabla_i R_{jk})(x) s_i s_j s_k,
\]
\[
\rho_4(\mathbf{s}) = \sum_{i,j,k,l=1}^{d} \left(
    -\frac{3}{5} \nabla^2_{ij} R_{kl}(x) 
    + \frac{1}{3} R_{ij}(x) R_{kl}(x) 
    - \frac{2}{15} \sum_{a,b=1}^{d} R_{iajb}(x) R_{kalb}(x)
\right) s_i s_j s_k s_l.
\]
Here, \( R_{ij}(x) \) denotes the Ricci curvature at \( x \), \( \nabla_i R_{jk}(x) \) denotes the covariant derivative of the Ricci tensor, and \( R_{iajb}(x) \) are the components of the Riemann curvature tensor at \( x \) in normal coordinates.
\end{thm}

Let \( P : (0, \delta) \times S^{d-1} \to B_0(\delta) \) denote the standard polar coordinate map in \( \mathbb{R}^d \), defined by
\(
P(s, \mathbf{v}) = s \mathbf{v},
\)
and
\[
\mathbf{y} : (0, \delta) \times S^{d-1} \to \mathcal{B}_x'(\delta)
\]
be the map defined by
\(
\mathbf{y} := \mathbf{x} \circ P.
\)
We refer to \( \mathbf{y} \) as the system of \emph{geodesic polar coordinates} on the punctured geodesic ball
\(
\mathcal{B}_x'(\delta).
\)  
Since
\(
\mathbf{y}^* dV = P^* (\mathbf{x}^* dV),
\)
and \( P(s, \mathbf{v}) = s \mathbf{v} \) is the standard polar coordinate map on \( \mathbb{R}^d \), we can express the pullback of the Riemannian volume form \( dV \) under \( \mathbf{y} \) using the Jacobian determinant of polar coordinates. This yields:
\[
\mathbf{y}^* dV = \rho(s \mathbf{v}) \, s^{d-1} \, ds \wedge d\sigma(\mathbf{v}),
\]
where \( d\sigma(\mathbf{v}) \) denotes the standard volume form on the unit sphere \( S^{d-1} \).Using the Taylor expansion of \( \rho \) and its corresponding representation in polar coordinates, we obtain the following result (assuming \( \rho \) is analytic):

\begin{cor}
The pullback of the volume form \( \mathbf{y}^* dV \) under geodesic polar coordinates admits the expansion
\[
\mathbf{y}^* dV = \left( \sum_{j=0}^{\infty} \frac{s^j}{j!} \rho_j(\mathbf{v}) \right) s^{d-1} \, ds \wedge d\sigma(\mathbf{v}),
\]
where each function \( \rho_j(\mathbf{v}) \) denotes the restriction of the homogeneous polynomial \( \rho_j(\mathbf{s}) \) of degree \( j \) to the unit sphere \( S^{d-1} \).
\end{cor}

\section{Asymptotic Expansion of the Gaussian integral operator}
Let \( M \) be a \( d \)-dimensional compact oriented Riemannian submanifold of \( \mathbb{R}^n \), and let \( x \in M \). Choose \( \delta > 0 \) such that \( \delta < \operatorname{inj}_x(M) \). Denote by \( \mathcal{B}_x(\delta) \) the geodesic ball centered at \( x \) with radius \( \delta \), as defined in the previous section.

Since \( M \setminus \mathcal{B}_x(\delta) \) is compact, there exists a point \( y_\delta \in M \setminus \mathcal{B}_x(\delta) \) such that
\[
\| y_\delta - x \| = \inf \{ \| y - x \| : y \in M \setminus \mathcal{B}_x(\delta) \}.
\]
We define \( m_\delta := \| y_\delta - x \| \), which satisfies \( m_\delta > 0 \) by construction.

For each function \( f \in C^\infty(M) \), the Gaussian kernel operator \( \mathcal{K}_\varepsilon \) can be decomposed as
\[
\mathcal{K}_\varepsilon f(x) = \int_{\mathcal{B}_x(\delta)} k_\varepsilon(x, y) f(y) \, dV_y + \int_{M \setminus \mathcal{B}_x(\delta)} k_\varepsilon(x, y) f(y) \, dV_y.
\]
To estimate the second term, observe that
\[
\left| \int_{M \setminus \mathcal{B}_x(\delta)} k_\varepsilon(x, y) f(y) \, dV_y \right| \leq (4\pi\varepsilon)^{-\frac{d}{2}}e^{-\frac{m_\delta^{2}}{4\varepsilon}} \operatorname{Vol}(M) \| f \|_\infty,
\]
where \( \operatorname{Vol}(M) \) denotes the total volume of \( M \). Since 
\[
\lim_{\varepsilon \to 0^+} \frac{(4\pi\varepsilon)^{-\frac{d}{2}}e^{-\frac{m_\delta^{2}}{4\varepsilon}}}{\varepsilon^k} = 0 \quad \text{for all } k \in \mathbb{N},
\]
it then follows that
\[
\int_{M \setminus \mathcal{B}_x(\delta)} k_\varepsilon(x, y) f(y) \, dV_y = o(\varepsilon^k)
\quad \text{for all } k \in \mathbb{N}.
\]

To estimate the first term, we rewrite the integral using geodesic normal coordinates 
\( \mathbf{x} : B_0(\delta) \to \mathcal{B}_x(\delta) \). This gives
\[
\int_{\mathcal{B}_x(\delta)} k_\varepsilon(x, y) f(y) \, dV_y 
= (4\pi\varepsilon)^{-\frac{d}{2}} \int_{B_0(\delta)} e^{-\frac{\| \mathbf{x}(\mathbf{s}) - x \|^2}{4\varepsilon}} 
f(\mathbf{x}(\mathbf{s})) \, \mathbf{x}^* dV_y.
\]
We now further assume that both the function \( \widetilde{f} = f \circ \mathbf{x} \) and the density function \( \rho \), associated with the pullback volume form
\[
\mathbf{x}^* dV_y = \rho(\mathbf{s}) \, ds_{1} \wedge \cdots \wedge ds_{d},
\]
are real analytic on \( B_0(\delta) \), and that their Taylor series converge uniformly on this domain.

To begin, we compute the Taylor expansion of the product \( \widetilde{f}(\mathbf{s}) \, \mathbf{x}^* dV_y \) in geodesic polar coordinates. Since
\[
\widetilde{f}(\mathbf{s}) = \sum_{k=0}^{\infty} \frac{s^k}{k!} \widetilde{f}_k(\mathbf{v}), \quad 
\mathbf{x}^* dV_y = \left( \sum_{j=0}^{\infty} \frac{s^j}{j!} \rho_j(\mathbf{v}) \right) s^{d-1} \, ds \wedge d\sigma(\mathbf{v}),
\]
their product admits the expansion
\[
\widetilde{f}(\mathbf{s}) \, \mathbf{x}^* dV_y = 
\left( \sum_{\ell=0}^{\infty} \frac{s^{\ell}}{\ell!} \alpha_{\ell}(\mathbf{v}) \right) s^{d-1} \, ds \wedge d\sigma(\mathbf{v}),
\]
where each coefficient \( \alpha_{\ell}(\mathbf{v}) \) is either zero or a homogeneous polynomial of degree \( \ell \), and is given by the convolution formula
\[
\alpha_{\ell}(\mathbf{v}) = \sum_{j=0}^{\ell} \binom{\ell}{j} \widetilde{f}_j(\mathbf{v}) \, \rho_{\ell - j}(\mathbf{v})
\quad \text{for each } \ell \geq 0.
\]

Now, we define the function \( q : B_0(\delta) \to \mathbb{R} \) by
\[
q(\mathbf{s}) = g(\mathbf{s}) - s^2,
\]
where \( g : B_0(\delta) \to \mathbb{R} \) is the squared Euclidean distance function defined in the previous section, and \( s = \sqrt{s_1^2 + \cdots + s_d^2} \). Then \( q \) is analytic at the origin and admits a Taylor expansion of the form
\[
q(\mathbf{s}) = \sum_{j=0}^{\infty} \frac{1}{j!} q_j(\mathbf{s}),
\]
which converges on the open ball \( B_0(\delta) \). Here, each \( q_j(\mathbf{s}) \) is the degree-\( j \) homogeneous component in the Taylor expansion of \( q \).  
By Lemma~\ref{tmetric}, the first few terms vanish:
\[
q_0(\mathbf{s}) = 0, \quad
q_1(\mathbf{s}) = 0, \quad
q_2(\mathbf{s}) = 0, \quad
q_3(\mathbf{s}) = 0.
\]
Since \( q(\mathbf{s}) = g(\mathbf{s}) - s^2 \), we may write
\[
e^{-\frac{\| \mathbf{x}(\mathbf{s}) - x \|^2}{4\varepsilon}} = e^{-\frac{s^2}{4\varepsilon}} \cdot e^{-\frac{q(\mathbf{s})}{4\varepsilon}}.
\]
Hence, it suffices to compute the Taylor expansion of the function \( e^{-\frac{q(\mathbf{s})}{4\varepsilon}} \) about \( \mathbf{s} = 0 \).  
To carry out this computation systematically, we introduce the partial exponential Bell polynomials, which provide a convenient framework for expressing the Taylor expansion of composite functions such as \( e^{-\frac{q(\mathbf{s})}{4\varepsilon}} \).

\begin{defn}
The \emph{partial exponential Bell polynomial} \( B_{m,k}(x_1, x_2, \dots, x_{m-k+1}) \) is defined for integers \( m \geq 1 \) and \( 1 \leq k \leq m \) by
\[
B_{m,k}(x_1, x_2, \dots, x_{m-k+1}) = \sum \frac{m!}{j_1! j_2! \cdots j_{m-k+1}!}
\left( \frac{x_1}{1!} \right)^{j_1}
\left( \frac{x_2}{2!} \right)^{j_2}
\cdots
\left( \frac{x_{m-k+1}}{(m-k+1)!} \right)^{j_{m-k+1}},
\]
where the sum is taken over all sequences of non-negative integers \( (j_1, j_2, \dots, j_{m-k+1}) \) satisfying the constraints
\[
\sum_{i=1}^{m-k+1} j_i = k \quad \text{and} \quad \sum_{i=1}^{m-k+1} i j_i = m.
\]
\end{defn}

It is established that the generating function for the partial exponential Bell polynomials can be expressed as follows:
\begin{equation}\label{gb}
\exp\left(u \sum_{j=1}^\infty x_j \frac{t^j}{j!} \right)
= 1 + \sum_{m=1}^\infty \frac{t^m}{m!} \sum_{k=1}^m u^k B_{m,k}(x_1, x_2, \dots, x_{m-k+1}).
\end{equation}
By substituting \( x_j = q_j(\mathbf{v}) \) and \( t = s \) into equation \eqref{gb}, we derive
\[
e^{-\frac{q(\mathbf{s})}{4\epsilon}} = 1 + \sum_{m=1}^\infty \frac{s^m}{m!} \sum_{k=1}^m u^k b_{m,k}(\mathbf{v}),
\]
where \( u = (-4\epsilon)^{-1} \), and
\(
b_{m,k}(\mathbf{v}) = B_{m,k}\big(q_1(\mathbf{v}), q_2(\mathbf{v}), \dots, q_{m-k+1}(\mathbf{v})\big).
\)
Since each \( q_i \) is either zero or a homogeneous polynomial of degree \( i \), it follows from the definition of the partial exponential Bell polynomials that each \( b_{m,k}(\mathbf{v}) \) is also either zero or a homogeneous polynomial of degree \( m \).
In particular, given that \( q_1(\mathbf{v}) = q_2(\mathbf{v}) = q_3(\mathbf{v}) = 0 \), any monomial involving these terms will vanish. Thus, a necessary and sufficient condition for \( b_{m,k}(\mathbf{v}) \) to be a nonzero homogeneous polynomial is that all terms in its expansion must involve only \( q_j \) with \( j \geq 4 \). This requirement translates to the inequality:
\(
m \geq 4k.
\)
Let us define \( \beta_0(\mathbf{v}, u) = 1 \), and for \( m \geq 1 \),
\[
\beta_m(\mathbf{v}, u) = \sum_{k=1}^{m} u^k \, b_{m,k}(\mathbf{v}).
\]
Notice that $b_{m,k}(\mf v)=0$ when $k>m/4.$ Then the exponential factor admits the following expansion
\[
e^{-\frac{q(\mathbf{s})}{4\varepsilon}} = \sum_{m=0}^\infty \frac{s^m}{m!} \, \beta_m(\mathbf{v}, u).
\]
Define
\[
a_\ell(\mathbf{v}, u) = \sum_{m=0}^\ell \binom{\ell}{m} \alpha_{\ell - m}(\mathbf{v}) \, \beta_m(\mathbf{v}, u).
\]
Since each \( \alpha_j(\mathbf{v}) \) is either zero or a homogeneous polynomial of degree \( j \), and each \( \beta_k(\mathbf{v}, u) \) is either zero or a homogeneous polynomial of degree \( k \), their product yields a finite sum of homogeneous polynomials of total degree \( \ell \).  
Moreover, by symmetry, the integral of any homogeneous polynomial of odd degree over the unit sphere \( S^{d-1} \) vanishes. Therefore,
\[
\int_{S^{d-1}} a_\ell(\mathbf{v}, u) \, d\sigma(\mathbf{v}) = 0 \qquad \text{whenever } \ell \text{ is odd}.
\]
It thus suffices to consider the case \( \ell = 2p \), where \( p \in \mathbb{N}_0 \). In that case, we have
\[
\int_{S^{d-1}} a_{2p}(\mathbf{v}, u) \, d\sigma(\mathbf{v}) = 
\int_{S^{d-1}} \alpha_{2p}(\mathbf{v}) \, d\sigma(\mathbf{v}) 
+ \sum_{m=1}^{2p} \binom{2p}{m} \int_{S^{d-1}} \alpha_{2p - m}(\mathbf{v}) \, \beta_m(\mathbf{v}, u) \, d\sigma(\mathbf{v}).
\]
We now define the scalar coefficients
\[
\eta_p = \frac{1}{\omega_d} \int_{S^{d-1}} \alpha_{2p}(\mathbf{v}) \, d\sigma(\mathbf{v}), \quad
w_{p,m,k} = \frac{1}{\omega_d} \int_{S^{d-1}} \alpha_{2p - m}(\mathbf{v}) \, b_{m,k}(\mathbf{v}) \, d\sigma(\mathbf{v}),
\]
for \( 1 \leq k \leq \left\lfloor \frac{m}{4} \right\rfloor \), where \( \omega_d = \operatorname{Vol}(S^{d-1}) \) denotes the volume of the unit sphere in \( \mathbb{R}^d \).  
Then we obtain
\[
\frac{1}{\omega_d} \int_{S^{d-1}} a_{2p}(\mathbf{v}, u) \, d\sigma(\mathbf{v})
= \eta_p + \sum_{m=1}^{2p} \sum_{k=1}^{m} \binom{2p}{m} (-4\varepsilon)^{-k} \, w_{p,m,k}.
\]

Since  
\[
\int_{B_0(\delta)} e^{-\frac{\| \mathbf{x}(\mathbf{s}) - x \|^2}{4\varepsilon}} f(\mathbf{x}(\mathbf{s})) \, \mathbf{x}^* dV_y
= \sum_{\ell=0}^{\infty} \frac{1}{\ell!} \int_{S^{d-1}} \left( \int_0^\delta e^{-\frac{s^2}{4\varepsilon}} s^{\ell + d - 1} \, ds \right) a_{\ell}(\mathbf{v}, u) \, d\sigma(\mathbf{v}),
\]
we obtain the expansion  
\[
\int_{\mathcal{B}_x(\delta)} k_\varepsilon(x, y) f(y) \, dV_y
= \sum_{p=0}^{\infty} \frac{ \omega_d c_p(\varepsilon) }{(2p)!}
\left(\eta_p + \sum_{m=1}^{2p} \sum_{k=1}^{m}
\binom{2p}{m} \frac{1}{(-4\varepsilon)^k} \, w_{p,m,k} \right),
\]
where  
\[
c_p(\varepsilon) = \frac{1}{(4\pi\varepsilon)^{\frac{d}{2}}} \int_0^\delta e^{- \frac{s^2}{4\varepsilon}} s^{2p + d - 1} \, ds.
\]

This power series can be reorganized as  
\[
\int_{\mathcal{B}_x(\delta)} k_\varepsilon(x, y) f(y) \, dV_y
= \sum_{q=0}^{\infty} \left( \frac{ \omega_d c_q(\varepsilon) }{(2q)!}\eta_{q}
+ \sum_{k=1}^{q} \sum_{m=4k}^{2q+2k} \frac{\omega_d c_{q+k}(\varepsilon)}{m!(2q+2k-m)!} (-4\varepsilon)^{-k} w_{q+k,m,k} \right).
\]

Next, we will use a sneaky trick to compute the asymptotic expansion.

\begin{lem}
For each nonnegative integer \( p \), we have
\[
\left| c_p(\varepsilon) - \frac{(4\varepsilon)^p}{2(\pi)^{\frac{d}{2}}} \, \Gamma\left( p + \frac{d}{2} \right) \right|
\leq
2^{p + \frac{d}{2}} \, e^{ -\frac{\delta^2}{8\varepsilon} } \cdot \frac{(4\varepsilon)^p}{2(\pi)^{\frac{d}{2}}} \, \Gamma\left( p + \frac{d}{2} \right).
\]
\end{lem}

\begin{proof}
Using a change of variables, we have
\[
\int_{0}^{\delta} e^{-\frac{s^{2}}{4\varepsilon}} s^{2p + d - 1} \, ds 
= \frac{(4\varepsilon)^{p + \frac{d}{2}}}{2} \int_{0}^{\frac{\delta^{2}}{4\varepsilon}} e^{-t} t^{p + \frac{d}{2} - 1} \, dt.
\]
Hence,
\[
c_p(\varepsilon) 
= \frac{(4\varepsilon)^p}{2(\pi)^{\frac{d}{2}}} \int_{0}^{\frac{\delta^{2}}{4\varepsilon}} e^{-t} t^{p + \frac{d}{2} - 1} \, dt 
= \frac{(4\varepsilon)^p}{2(\pi)^{\frac{d}{2}}} \left( \Gamma\left(p + \frac{d}{2} \right) - \int_{\frac{\delta^2}{4\varepsilon}}^{\infty} e^{-t} t^{p + \frac{d}{2} - 1} \, dt \right).
\]
Therefore,
\begin{align*}
\left| c_p(\varepsilon) - \frac{(4\varepsilon)^p}{2\pi^{\frac{d}{2}}} \Gamma\left( p + \frac{d}{2} \right) \right| 
&= \frac{(4\varepsilon)^p}{2\pi^{\frac{d}{2}}} \int_{\frac{\delta^2}{4\varepsilon}}^{\infty} e^{-t} t^{p + \frac{d}{2} - 1} \, dt \\
&\leq \frac{(4\varepsilon)^p}{2\pi^{\frac{d}{2}}} \, e^{ -\frac{\delta^2}{8\varepsilon} } \int_{\frac{\delta^2}{4\varepsilon}}^{\infty} e^{-t/2} t^{p + \frac{d}{2} - 1} \, dt \\
&\leq \frac{(4\varepsilon)^p}{2\pi^{\frac{d}{2}}} \, e^{ -\frac{\delta^2}{8\varepsilon} } \int_{0}^{\infty} e^{-t/2} t^{p + \frac{d}{2} - 1} \, dt \\
&=  2^{p + \frac{d}{2}} \, e^{ -\frac{\delta^2}{8\varepsilon} } \cdot \frac{(4\varepsilon)^p}{2\pi^{\frac{d}{2}}} \, \Gamma\left( p + \frac{d}{2} \right).
\end{align*}
Simplifying the constants gives the desired bound.
\end{proof}

As a consequence, we obtain the estimate  
\[
\left| \omega_d c_p(\varepsilon) - \left( \frac{d}{2} \right)_p(4\varepsilon)^p \right|
\leq
2^{p + \frac{d}{2}} \left( \frac{d}{2} \right)_p (4\varepsilon)^p \, e^{- \frac{\delta^2}{8\varepsilon}},
\]
where \( (q)_n = \prod_{j=0}^{n-1} (q + j) \) for \( n \geq 1 \), and \( (q)_0 = 1 \) denotes the Pochhammer symbol.
Therefore,
\begin{align*}
\int_{\mathcal{B}_x(\delta)} k_\varepsilon(x, y) f(y) \, dV_y
&= \sum_{q=0}^{\infty} \varepsilon^q \left[ \frac{4^q  \left( \frac{d}{2} \right)_q}{(2q)!}\eta_{q}
+ \sum_{k=1}^{q} \sum_{m=4k}^{2q+2k} \frac{(-1)^k 4^q \left( \frac{d}{2} \right)_{q+k}}{m!(2q+2k - m)!} w_{q+k, m, k} \right] \\
&\quad + R_\varepsilon(\delta),
\end{align*}
where the remainder term \( R_\varepsilon(\delta) \) is given by
\[
R_\varepsilon(\delta) = \sum_{q=0}^{\infty} \left[ \frac{ \omega_d c_q(\varepsilon) - \left( \frac{d}{2} \right)_q(4\varepsilon)^q }{(2q)!}\eta_{q}
+ \sum_{k=1}^{q} \sum_{m=4k}^{2q+2k} \frac{ \omega_d c_{q+k}(\varepsilon) -\left( \frac{d}{2} \right)_{q+k}(4\varepsilon)^{q+k} }{m!(2q+2k - m)!} (-4\varepsilon)^{-k} w_{q+k, m, k} \right].
\]

We estimate the remainder as follows:
\begin{align*}
\left| \frac{ \omega_d c_q(\varepsilon) - \left( \frac{d}{2} \right)_q (4\varepsilon)^q }{(2q)!}\eta_{q} \right|
&\leq \varepsilon^q e^{- \frac{\delta^2}{8\varepsilon}} \cdot \frac{ 2^{3q + \frac{d}{2}} \left( \frac{d}{2} \right)_q}{(2q)!}|\eta_{q}|, \\
\left| \frac{ \omega_d c_{q+k}(\varepsilon) - \left( \frac{d}{2} \right)_{q+k}(4\varepsilon)^{q+k} }{m!(2q+2k - m)!} (-4\varepsilon)^{-k} w_{q+k, m, k} \right|
&\leq \varepsilon^q e^{- \frac{\delta^2}{8\varepsilon}} \cdot \frac{ 2^{3q + k + \frac{d}{2}} \left( \frac{d}{2} \right)_{q+k}}{m!(2q+2k - m)!} \left| w_{q+k, m, k} \right|,
\end{align*}
and hence
\[
|R_{\varepsilon}(\delta)| \leq e^{- \frac{\delta^2}{8\varepsilon}} \, \xi(\varepsilon),
\]
where \( \xi(\varepsilon) = \sum_{q=0}^{\infty} \varepsilon^q \xi_q \) is a convergent power series with
\[
\xi_q = \frac{ 2^{3q + \frac{d}{2}} \left( \frac{d}{2} \right)_q }{(2q)!}|\eta_{q}|
+ \sum_{k=1}^{q} \sum_{m=4k}^{2q+2k} \frac{ 2^{3q + k + \frac{d}{2}} \left( \frac{d}{2} \right)_{q+k} }{(2q+2k)!}\binom{2q+2k}{m} \left| w_{q+k, m, k} \right|.
\]
This also implies that
\[
\lim_{\varepsilon \to 0^+} \frac{R_{\varepsilon}(\delta)}{\varepsilon^k} = 0
\]
for all \( k \in \mathbb{N} \). We conclude that the asymptotic expansion of \( \mathcal{K}_\varepsilon f \) as \( \varepsilon \to 0^+ \) is given by:

\begin{thm}
For each \( q \geq 0 \), define
\[
a_q(x) = \frac{4^q \left( \frac{d}{2} \right)_q}{(2q)!} \eta_q
+ \sum_{k=1}^{q} \sum_{m=4k}^{2q+2k} \frac{(-1)^k 4^q \left( \frac{d}{2} \right)_{q+k}}{m!(2q+2k - m)!} w_{q+k, m, k}.
\]
Then the operator \( \mathcal{K}_\varepsilon \) admits the asymptotic expansion
\[
(\mathcal{K}_\varepsilon f)(x) \sim \sum_{q=0}^{\infty} a_q(x) \, \varepsilon^q \quad \text{as } \varepsilon \to 0^+.
\]
\end{thm}
\begin{proof}
We sketch the idea of the proof. Fix \( x \in M \), and let \( \mathbf{x}: B_0(\delta) \subset \mathbb{R}^d \to M \) be a system of geodesic normal coordinates centered at \( x \), where \( \delta > 0 \) is smaller than the injectivity radius at \( x \). We decompose the integral defining \( \mathcal{K}_\varepsilon f(x) \) as
\[
(\mathcal{K}_\varepsilon f)(x) = \frac{1}{(4\pi\varepsilon)^{\frac{d}{2}}} \left[ \int_{\mathcal{B}_x(\delta)} e^{ -\frac{ \| x - y \|^2 }{4\varepsilon} } f(y) \, dV_y + \int_{M \setminus \mathcal{B}_x(\delta)} e^{ -\frac{ \| x - y \|^2 }{4\varepsilon} } f(y) \, dV_y \right].
\]

The second term is exponentially small: there exists a constant \( C > 0 \) such that
\[
\int_{M \setminus \mathcal{B}_x(\delta)} e^{ -\frac{ \| x - y \|^2 }{4\varepsilon} } f(y) \, dV_y = O(e^{-C/\varepsilon}) = o(\varepsilon^{N+1}),
\]
for every \( N \), since the integrand decays faster than any polynomial in \( \varepsilon \).

We now focus on the integral over \( \mathcal{B}_x(\delta) \), transforming it into Euclidean coordinates via the normal chart:
\[
I_\varepsilon(x) = \frac{1}{(4\pi\varepsilon)^{\frac{d}{2}}} \int_{B_0(\delta)} e^{- \frac{ \| \mathbf{x}(\mathbf{s}) - x \|^2 }{4\varepsilon} } \, \widetilde{f}(\mathbf{s}) \, \rho(\mathbf{s}) \, d\mathbf{s},
\]
where \( \widetilde{f} = f \circ \mathbf{x} \), \( \rho(\mathbf{s}) \) is the density of the volume form \( \mathbf{x}^* dV \), and \( d\mathbf{s} = ds_1 \cdots ds_d \).

Since \( \| \mathbf{x}(\mathbf{s}) - x \|^2 = s^2 + q(\mathbf{s}) \), where \( q(\mathbf{s})\)                                                                                                                                                                  consists of quartic and higher-order terms in \( \mathbf{s} \), we expand the exponential as
\[
e^{ - \frac{ \| \mathbf{x}(\mathbf{s}) - x \|^2 }{4\varepsilon} } 
= e^{ - \frac{s^2}{4\varepsilon} } \left( 1 + \sum_{k=1}^{N} \frac{ (-1)^k }{k!} \left( \frac{ q(\mathbf{s}) }{4\varepsilon} \right)^k + R_N(\mathbf{s},\varepsilon) \right),
\]
where \( R_N(\mathbf{s}, \varepsilon) = O\left( \frac{s^{4N+4}}{\varepsilon^{N+1}} \right) \).

Likewise, we Taylor expand \( \widetilde{f}(\mathbf{s}) \rho(\mathbf{s}) \) at \( \mathbf{s} = 0 \) up to order \( 2N \):
\[
\widetilde{f}(\mathbf{s}) \rho(\mathbf{s}) = \sum_{j \le 2N} \frac{ H_{j}(\widetilde{f} \rho)}{j!}(0)  s^{j} + O(s^{2N+1}).
\]

Multiplying the expansions and collecting terms of total degree \( 2q \), we integrate term by term over \( \mathbf{s} \in B_0(\delta) \). Using spherical coordinates \( \mathbf{s} = s \mathbf{v} \), with \( \mathbf{v} \in S^{d-1} \), and the fact that
\[
\int_{B_0(\delta)} e^{ -\frac{s^2}{4\varepsilon} } s^{2q+d-1} \, ds \, d\sigma(\mathbf{v}) = c_q(\varepsilon)\, \omega_d,
\]
we find that the coefficient of \( \varepsilon^q \) arises from the terms \( \eta_q \) involving \( s^{2q} \) in the expansion of \( \widetilde{f}(\mathbf{s}) \rho(\mathbf{s}) \), as well as interaction terms involving \( q(\mathbf{s}) \), contributing correction terms \( w_{q+k,m,k} \) from higher-order interactions.
Since \( c_q(\varepsilon) \sim \varepsilon^{q} \) up to a constant multiple, we obtain, up to order \( N \),
\[
I_\varepsilon(x) = \sum_{q=0}^{N} \left( \frac{ \omega_d c_q(\varepsilon) }{(2q)!} \eta_q 
+ \sum_{k=1}^{q} \sum_{m=4k}^{2q+2k} \frac{ \omega_d c_{q+k}(\varepsilon) }{ m!(2q+2k - m)! } 
(-4\varepsilon)^{-k} w_{q+k, m, k} \right) + o(\varepsilon^{N+1}).
\]

By the previous lemma, this expression can be written as
\[
\sum_{q=0}^{N} a_q(x)\, \varepsilon^q + o(\varepsilon^{N+1}).
\]
Therefore, we conclude that
\[
(\mathcal{K}_\varepsilon f)(x) = I_\varepsilon(x) + o(\varepsilon^{N+1}) = \sum_{q=0}^{N} a_q(x) \varepsilon^q + o(\varepsilon^{N+1}),
\]
as required. Since this holds for arbitrary \( N \), the result follows.
\end{proof}

We remark that the quantities \( \eta_q \) and \( w_{q+k, m, k} \) are averages of homogeneous polynomials over the unit sphere \( S^{d-1} \), and hence can be computed using spherical moment integrals. For example, see G. B. Folland, \cite{F}.

\begin{thm}
For each multi-index \( \alpha = (\alpha_1, \dots, \alpha_d) \), the average of the monomial \( \mathbf{v}^{\alpha} = v_1^{\alpha_1} \cdots v_d^{\alpha_d} \) over the unit sphere \( S^{d-1} \subset \mathbb{R}^d \) is given by
\[
\frac{1}{\omega_d} \int_{S^{d-1}} \mathbf{v}^{\alpha} \, d\sigma(\mathbf{v}) =
\begin{cases}
0 & \text{if any } \alpha_i \text{ is odd}, \\
\displaystyle 2 \cdot \frac{ \prod_{i=1}^{d} \Gamma\left( \frac{\alpha_i + 1}{2} \right) }{ \Gamma\left( \frac{|\alpha| + d}{2} \right) } & \text{if all } \alpha_i \text{ are even}.
\end{cases}
\]
\end{thm}

We will compute \( a_0(x) \) and \( a_1(x) \) using the above formulas. First, let us compute \( a_0(x) = \eta_0 \). Since \( \alpha_0(\mathbf{v}) = \widetilde{f}_0(\mathbf{v}) \rho_0(\mathbf{v}) \), and
\[
\widetilde{f}_0(\mathbf{v}) = \widetilde{f}(0), \quad \rho_0(\mathbf{v}) = 1,
\]
we have
\[
\eta_0 = \frac{1}{\omega_d} \int_{S^{d-1}} \alpha_0(\mathbf{v}) \, d\sigma(\mathbf{v}) 
= \frac{1}{\omega_d} \int_{S^{d-1}} \widetilde{f}(0) \, d\sigma(\mathbf{v}) 
= \widetilde{f}(0) = f(x).
\]
We conclude that
\[
a_0(x) = f(x).
\]

Now, let us compute \( a_1(x) \). By the previous result,
\[
a_1(x) = d \eta_1 - \frac{d(d+2)}{4!} w_{2,4,1}.
\]
It therefore suffices to compute \( \eta_1 \) and \( w_{2,4,1} \):
\[
\eta_1 = \frac{1}{\omega_d} \int_{S^{d-1}} \alpha_2(\mathbf{v}) \, d\sigma(\mathbf{v}), \quad
w_{2,4,1} = \frac{1}{\omega_d} \int_{S^{d-1}} \alpha_0(\mathbf{v}) \, b_{4,1}(\mathbf{v}) \, d\sigma(\mathbf{v}).
\]
Since
\[
\alpha_2(\mathbf{v}) = \widetilde{f}_2(\mathbf{v}) \rho_0(\mathbf{v}) + 2 \widetilde{f}_1(\mathbf{v}) \rho_1(\mathbf{v}) + \widetilde{f}_0(\mathbf{v}) \rho_2(\mathbf{v}),
\]
and \( \widetilde{f}_0(\mathbf{v}) = f(x) \), by Theorem~\ref{g1}, we obtain
\[
\eta_1 = \frac{1}{\omega_d} \left( \int_{S^{d-1}} \widetilde{f}_2(\mathbf{v}) \, d\sigma(\mathbf{v}) + f(x) \int_{S^{d-1}} \rho_2(\mathbf{v}) \, d\sigma(\mathbf{v}) \right).
\]
Since
\begin{align*}
\widetilde{f}_2(\mathbf{v}) &= \sum_{i=1}^{d} \frac{\partial^2 \widetilde{f}}{\partial s_i^2}(0) \, v_i^2
+ 2 \sum_{1 \leq i < j \leq d} \frac{\partial^2 \widetilde{f}}{\partial s_i \partial s_j}(0) \, v_i v_j, \\
\rho_2(\mathbf{v}) &=  -\frac{1}{3} \sum_{i=1}^{d} R_{ii}(x) \, v_i^2
- \frac{1}{3} \sum_{1 \leq i < j \leq d} R_{ij}(x) \, v_i v_j,
\end{align*}
and using the identity
\[
\frac{1}{\omega_d} \int_{S^{d-1}} v_i^2 \, d\sigma(\mathbf{v}) = \frac{1}{d},
\]
we obtain
\begin{align*}
\frac{1}{\omega_d} \int_{S^{d-1}} \widetilde{f}_2(\mathbf{v}) \, d\sigma(\mathbf{v}) &= -\frac{1}{d} (\Delta f)(x), \\
\frac{1}{\omega_d} \int_{S^{d-1}} \rho_2(\mathbf{v}) \, d\sigma(\mathbf{v}) &= -\frac{1}{3d} R(x),
\end{align*}
where
\[
(\Delta f)(x) = -\sum_{i=1}^{d} \frac{\partial^2 \widetilde{f}}{\partial s_i^2}(0), \quad
R(x) = \sum_{i=1}^{d} R_{ii}(x),
\]
and \( R(x) \) denotes the scalar curvature of \( M \) at the point \( x \). Therefore, we conclude that
\[
\eta_1 = \frac{1}{d} \left( -(\Delta f)(x) - \frac{1}{3} R(x) f(x) \right).
\]

Now, let us compute \( w_{2,4,1} \). Since \( \alpha_0(\mathbf{v}) = f(x) \), we have
\[
w_{2,4,1} = \frac{f(x)}{\omega_d} \int_{S^{d-1}} b_{4,1}(\mathbf{v}) \, d\sigma(\mathbf{v}).
\]
Since \( B_{m,1}(x_1, \dots, x_m) = x_m \), it follows that
\[
b_{4,1}(\mathbf{v}) = q_4(\mathbf{v}) = -2 \|\mf B_x(\mathbf{v}_x, \mathbf{v}_x) \|^2.
\]
Thus,
\[
w_{2,4,1} = -2 \frac{f(x)}{\omega_d} \int_{S^{d-1}} \|\mf B_x(\mathbf{v}_x, \mathbf{v}_x) \|^2 \, d\sigma(\mathbf{v}).
\]
For \( \mathbf{v} = (v_1, \dots, v_d) \in S^{d-1} \), we write \( \mathbf{v}_x = \sum_{i=1}^d v_i (\mathbf{e}_i)_x \). Then
\[
\|\mf B_x(\mathbf{v}_x, \mathbf{v}_x) \|^2 = \sum_{i,j,k,l=1}^d b_{ijkl}(x) \, v_i v_j v_k v_l,
\]
where
\[
b_{ijkl}(x) = \left\langle \mf B_x((\mathbf{e}_i)_x, (\mathbf{e}_j)_x), \,\mf B_x((\mathbf{e}_k)_x, (\mathbf{e}_l)_x) \right\rangle.
\]
Using the known spherical moment integrals:
\[
\frac{1}{\omega_d} \int_{S^{d-1}} v_i^2 v_j^2 \, d\sigma(\mathbf{v}) =
\begin{cases}
\frac{1}{d(d+1)} & \text{if } i \ne j, \\
\frac{3}{d(d+2)} & \text{if } i = j,
\end{cases}
\]
we obtain
\[
\frac{1}{\omega_d} \int_{S^{d-1}} \sum_{i,j,k,l=1}^d b_{ijkl}(x) \, v_i v_j v_k v_l \, d\sigma(\mathbf{v})
= \frac{1}{d(d+2)} \sum_{i,j=1}^d \left( b_{iijj}(x) + b_{ijij}(x) + b_{ijji}(x) \right).
\]
Finally, by the Gauss equation, we conclude that
\[
\frac{1}{\omega_d} \int_{S^{d-1}} \|\mf B_x(\mathbf{v}_x, \mathbf{v}_x) \|^2 \, d\sigma(\mathbf{v})
= \frac{1}{d(d+2)} \left( 3 \left\| \sum_{i=1}^d\mf B_x((\mathbf{e}_i)_x, (\mathbf{e}_i)_x) \right\|^2 - 2 R(x) \right),
\]
where \( R(x) \) denotes the scalar curvature at \( x \).  Since the mean curvature vector of \( M \) at \( x \) is defined by
\[
\mathbf{H}(x) = \frac{1}{d} \sum_{i=1}^d \mathsf{B}_x((\mathbf{e}_i)_x, (\mathbf{e}_i)_x),
\]
it follows that
\[
a_1(x) = -(\Delta f)(x) + \frac{f(x)}{4} \left( d^2 \left\| \mathbf{H}(x) \right\|^2 - 2 R(x) \right),
\]
where \( \Delta \) denotes the Laplace–Beltrami operator on \( M \), \( \mathsf{B}_x \) is the second fundamental form, and \( R(x) \) is the scalar curvature at \( x \). This expression agrees with the result established in \cite{HAL}.

For \( a_2(x) \), the result involves the derivatives of various curvature quantities, including the full Riemannian curvature tensor, Ricci curvature, and scalar curvature. Additionally, it incorporates the derivatives of the second fundamental form, which in turn affects the derivatives of the mean curvature. To derive the expression for \( a_2(x) \), one can employ the Taylor expansion of the volume density as articulated in Gray's work, alongside the spherical moment integrals.

\section{Hypersurfaces in Euclidean Spaces and their Equicurved Points}
In this section, we employ the asymptotic expansion of the Gaussian integral operator defined on a compact oriented hypersurface \( M \subset \mathbb{R}^{d+1} \), which is equipped with a global unit normal vector field \( \nu \). This analysis facilitates the classification of points that meet specific curvature criteria, termed \emph{equicurved points}. To lay the groundwork for this exploration, we will begin by reviewing essential concepts related to hypersurfaces in Euclidean space.

The \emph{Gauss map} associated with the oriented hypersurface \( (M, \nu) \) is defined as the smooth map
\(
G : M \to S^d 
\)
given by
\[
G(x) = T_{x,0} \nu(x),
\]
where \( T_{x,0} : T_x \mathbb{R}^{d+1} \to \mathbb{R}^{d+1} \) denotes the parallel translation (i.e., vector translation) from the point \( x \) to the origin. Explicitly, the map \( T_{x, 0} \) sends a tangent vector \( \mathbf{v}_x \in T_x \mathbb{R}^{d+1} \) to the corresponding vector \( \mathbf{v} \in \mathbb{R}^{d+1} \), via the identification \( T_0 \mathbb{R}^{d+1} \cong \mathbb{R}^{d+1} \).  
Accordingly, for each \( x \in M \), the unit normal vector \( \nu(x) \in T_x \mathbb{R}^{d+1} \) satisfies
\(
\nu(x) = (G(x))_x,
\)
where \( G(x) \in \mathbb{R}^{d+1} \cong T_0 \mathbb{R}^{d+1} \) and \( (G(x))_x \in T_x \mathbb{R}^{d+1} \) denotes the corresponding tangent vector at the point \( x \). 

The \emph{shape operator} \( S_x \) at a point \( x \in M \) is the linear map
\(
S_x :\; T_x M \to T_x \mathbb{R}^{d+1}
\)
defined by sending a tangent vector \( \mathbf{v}_x \in T_x M \) to
\[
S_x(\mathbf{v}_x) = -(T_{G(x), x} \circ dG_x)(\mathbf{v}_x).
\]
Here, \( dG_x \; T_x M \to T_{G(x)} S^d \subset \mathbb{R}^{d+1} \) is the differential of the Gauss map \( G \), and \( T_{G(x), x} \) is the translation map that identifies the tangent space \( T_{G(x)} S^d \) with \( T_x M \).

The image of \( S_x \) lies in \( T_x M \), and \( S_x \) is a self-adjoint linear operator with respect to the Riemannian metric induced on \( M \). By the spectral theorem, \( S_x \) admits a complete set of eigenpairs \( \{ (\kappa_i(x), (\mathbf{e}_i)_x) \; 1 \leq i \leq d \} \), where \( \kappa_i(x) \) is an eigenvalue of \( S_x \) with corresponding eigenvector \( (\mathbf{e}_i)_x \). The eigenvalues are ordered as
\[
\kappa_1(x) \geq \kappa_2(x) \geq \cdots \geq \kappa_d(x).
\]
The vectors \( \{ (\mathbf{e}_i)_x \}_{i=1}^d \) form an orthonormal basis of \( T_x M \).

We refer to \( \kappa_i(x) \) as the \( i \)-th principal curvature of \( M \) at \( x \), and to \( (\mathbf{e}_i)_x \) as the corresponding \( i \)-th principal direction.

We define the bilinear form
\(
B_x :\; T_x M \times T_x M \to \mathbb{R}
\)
by
\[
B_x(\mathbf{v}_x, \mathbf{w}_x) = \langle S_x(\mathbf{v}_x), \mathbf{w}_x \rangle.
\]
Then we have
\[
B_x(\mathbf{v}_x, \mathbf{w}_x) = \sum_{i=1}^{d} \kappa_i(x) \langle \mathbf{v}_x, (\mathbf{e}_i)_x \rangle \langle \mathbf{w}_x, (\mathbf{e}_i)_x \rangle
\]
for all \( \mathbf{v}_x, \mathbf{w}_x \in T_x M \). In this setting, the second fundamental form \( \mathbf{B}_x \) at the point \( x \in M \) is given by
\[
\mathbf{B}_x(\mathbf{v}_x, \mathbf{w}_x) = B_x(\mathbf{v}_x, \mathbf{w}_x) \, \nu(x)
\]
where \( \nu(x) \) is the unit normal vector at \( x \). In particular,
\[
\| \mathbf{B}_x(\mathbf{v}_x, \mathbf{w}_x) \|^2 = | B_x(\mathbf{v}_x, \mathbf{w}_x) |^2.
\]

\begin{defn}
Let \( (M, \nu) \) be a smooth oriented hypersurface in \( \mathbb{R}^{d+1} \), and let \( \kappa_1(x), \dots, \kappa_d(x) \) denote the principal curvatures of \( M \) at a point \( x \in M \). The \emph{\( i \)-th mean curvature} of \( M \) at \( x \) is defined as
\[
H_i(x) = \frac{1}{\binom{d}{i}} \, e_i\big(\kappa_1(x), \dots, \kappa_d(x)\big)
\]
where \( e_i \) denotes the \( i \)-th elementary symmetric polynomial in \( d \) variables. Specifically, \( e_i(\kappa_1, \dots, \kappa_d) \) is the sum of all products of \( i \) distinct principal curvatures.

In particular, the (first) \emph{mean curvature} of \( M \) at \( x \), denoted by \( H(x) \), is defined as
\[
H(x) = H_1(x).
\]
\end{defn}

Note that the mean curvature vector of \( M \) at \( x \) is given by
\[
\mathbf{H}(x) = H(x) \, \nu(x),
\]
where \( H(x) \) is the mean curvature and \( \nu(x) \) is the unit normal vector at \( x \).

Note that any hypersurface is locally orientable, meaning that there exists a local unit normal vector field \( \nu \) defined on a neighborhood of each point in \( M \). Consequently, one can define a \emph{local Gauss map} by associating to each point the direction of the corresponding unit normal vector. In particular, the \emph{principal curvatures} of \( M \) can be defined locally as the eigenvalues of the shape operator determined by this local Gauss map.

Let \( (\mathbf{e}_i)_x \) represent the \( i \)-th principal direction of the manifold \( M \) at the point \( x \in M \). The set \( \{ (\mathbf{e}_i)_x \; 1 \leq i \leq d \} \) constitutes an ordered orthonormal basis of the tangent space \( T_x M \). This basis induces a canonical system of normal coordinates centered at the point \( x \). For any \( 0 < \delta < \operatorname{inj}_x(M) \), this basis provides a standard choice for the normal coordinate parametrization
\[
\mathbf{x} :\; B_0(\delta) \to \mathcal{B}_x(\delta)
\]
as previously defined. For any \( \mathbf{v}_{x} = \sum_{i=1}^{d} v_i (\mathbf{e}_i)_x \), the second fundamental form at \( x \) satisfies
\[
B_x(\mathbf{v}_{x}, \mathbf{v}_{x}) = \sum_{i=1}^{d} \kappa_i(x) v_i^2
\]
where \( \kappa_i(x) \) denotes the \( i \)-th principal curvature of \( M \) at \( x \). It follows that
\[
\| \mathbf{B}_x(\mathbf{v}_{x}, \mathbf{v}_{x}) \|^2 = \left( \sum_{i=1}^{d} \kappa_i(x) v_i^2 \right)^2.
\]
Utilizing the spherical moment integrals and the identity \( R(x) = 2 e_2(\kappa_1(x), \dots, \kappa_d(x)) \), where \( e_2 \) denotes the second elementary symmetric polynomial of the principal curvatures, the integral of \( b_{4,1}(\mathbf{v}) \) over the unit sphere \( S^{d-1} \subset \mathbb{R}^d \) is expressed as follows
\begin{align*}
\frac{1}{\omega_d} \int_{S^{d-1}} b_{4,1}(\mathbf{v}) \, d\sigma(\mathbf{v})
&= \frac{-2}{\omega_d} \int_{S^{d-1}} \left( \sum_{i=1}^d \kappa_i(x)^2 v_i^4 + 2 \sum_{1 \leq i < j \leq d} \kappa_i(x) \kappa_j(x) v_i^2 v_j^2 \right) d\sigma(\mathbf{v}) \\
&= -\frac{2}{d(d+2)} \left( 3 \sum_{i=1}^d \kappa_i(x)^2 + 2 \sum_{1 \leq i < j \leq d} \kappa_i(x) \kappa_j(x) \right) \\
&= -\frac{2}{d(d+2)} \left( 3 e_1(\kappa_1(x), \dots, \kappa_d(x))^2 - 4 e_2(\kappa_1(x), \dots, \kappa_d(x)) \right) \\
&= -\frac{2}{d(d+2)} \left( 3 d^2 H^2(x) - 2 R(x) \right).
\end{align*}
Here,\( H(x) \) represents the mean curvature of \( M \) at \( x \), and \( R(x) \) is the scalar curvature at \( x \).
Inserting this expression into the asymptotic expansion results in:
\[
(\mathcal{K}_\varepsilon f)(x) \sim f(x) + \left( -\Delta f(x) + \frac{1}{4} \left( d^2 H^2(x) - 2 R(x) \right) f(x) \right) \varepsilon + \cdots
\]
where \( \Delta \) represents the Laplace–Beltrami operator on \( M \). This outcome aligns with the expression derived in the previous section when we apply the identity \[ \| \mathbf{H}(x) \|^2 = H(x)^2 \] directly.

\begin{defn}
Let \( x \) be a point on a hypersurface \( M \subset \mathbb{R}^{d+1} \). We say that \( M \) is \emph{equicurved at \( x \)} (or that \( x \) is an \emph{equicurved point} of \( M \)) if the following identity holds:
\[
d^2 H(x)^2 = 2 R(x),
\]
where \( H(x) \) denotes the mean curvature and \( R(x) \) the scalar curvature at the point \( x \). Moreover, \( M \) is called an \emph{equicurved hypersurface} if it is equicurved at every point \( x \in M \).
\end{defn}

It follows from this definition that:

\begin{thm}\label{ck}
Let \( M \subset \mathbb{R}^{d+1} \) be a compact hypersurface, and let \( x \in M \). Then \( x \) is an equicurved point if and only if the following limit holds:
\[
\lim_{\varepsilon \to 0^+} \frac{f(x) - (\mathcal{K}_\varepsilon f)(x)}{\varepsilon} = \Delta f(x)
\]
for every smooth function \( f : M \to \mathbb{R} \), where \( \Delta \) denotes the Laplace–Beltrami operator on \( M \).
\end{thm}

This theorem offers a practical and computable criterion for identifying equicurved points on a hypersurface by analyzing the asymptotic behavior of the Gaussian integral operator. Specifically, at an equicurved point, the Gaussian integral operator asymptotically behaves similarly to the heat operator.
These insights naturally prompt several fundamental questions: Can a compact hypersurface in Euclidean space possess an equicurved point? If so, is there a compact hypersurface that is equicurved at every point?
To explore these questions, we will begin by examining some fundamental properties of equicurved points on hypersurfaces.

Note that
\[
d^2 H^2 - 2 R = \left( e_1(\kappa_1, \dots, \kappa_d) \right)^2 - 4 e_2(\kappa_1, \dots, \kappa_d),
\]
where \( e_1 \) and \( e_2 \) denote the first and second elementary symmetric polynomials of the principal curvatures, respectively.
It follows that the identity \( d^2 H^2 - 2 R = 0 \) holds if and only if
\begin{equation}\label{e1e2}
\left( e_1(\kappa_1, \dots, \kappa_d) \right)^2 = 4 e_2(\kappa_1, \dots, \kappa_d).
\end{equation}

This identity allows us to derive a useful geometric consequence concerning minimal or scalar-flat hypersurfaces:

\begin{prop}
Let \( M \subset \mathbb{R}^{d+1} \) be a hypersurface, and let \( x \in M \) be an equicurved point. If \( M \) is minimal at \( x \), i.e., \( H(x) = 0 \), or if \( M \) is scalar flat at \( x \), i.e. \( R(x) = 0 \), then all principal curvatures vanish at \( x \).
\end{prop}
\begin{proof}
Suppose \( M \) is equicurved at \( x \), so that the identity \eqref{e1e2} holds.

If \( H(x) = 0 \), then \( e_1 = 0 \), and hence \eqref{e1e2} implies \( e_2 = 0 \).

If \( R(x) = 0 \), then \( e_2 = 0 \), and again by \eqref{e1e2}, we conclude that \( e_1 = 0 \).

In either case, we compute:
\[
\sum_{i=1}^d \kappa_i^2 = e_1^2 - 2e_2 = 0,
\]
which implies \( \kappa_i = 0 \) for all \( 1 \leq i \leq d \). Thus, all principal curvatures vanish at \( x \).
\end{proof}

The condition \eqref{e1e2} is also equivalent to the algebraic identity:
\begin{equation}\label{ele3}
\sum_{i=1}^{d} \kappa_i^2 - 2 \sum_{1 \leq i < j \leq d} \kappa_i \kappa_j = 0.
\end{equation}

When \( d = 2 \), the identity simplifies to \( (\kappa_2 - \kappa_1)^2 = 0 \), which implies that \(\kappa_1 = \kappa_2 \). Therefore, for \( d = 2 \), a point \( x \) in the manifold \( M \) is equicurved if and only if it is umbilical. According to a classical result in the differential geometry of surfaces (see, for instance, Do Carmo  \cite{C}), any compact umbilical surface in \( \mathbb{R}^3 \) must be a round sphere. Consequently, any compact equicurved surface in \( \mathbb{R}^3 \) is necessarily a sphere.
We encourage the reader to verify that Theorem~\ref{ck} holds at every point of \( M = S^2 \). This can be easily confirmed through direct computation using the explicit formula for geodesics on the unit sphere \( S^2 \).
On the other hand, it is possible to construct compact nonspherical surfaces that contain umbilical points, and therefore equicurved points. A classical example of this is a spheroid in \( \mathbb{R}^3 \). While a spheroid is not a sphere, it has finitely many umbilical (and thus equicurved) points.
We have therefore answered the previous questions in the case where \( d = 2 \).

When \( d \geq 3 \), an equicurved point of a hypersurface may not be umbilical. Suppose \( x \in M \) is an equicurved point. If \( x \) is also umbilical, then all principal curvatures at \( x \) would be equal, i.e.,
\[
\kappa_1(x) = \kappa_2(x) = \cdots = \kappa_d(x) = \kappa.
\]
Substituting into (\ref{ele3}), we obtain
\[
(d\kappa)^2 = 4 \binom{d}{2} \kappa^2.
\]
This simplifies to
\(
d(2 - d)\kappa^2 = 0.
\)
Since \( d \geq 3 \), it follows that \( k = 0 \). Thus, the point \( x \) must be a flat point of the hypersurface \( M \). We obtain the following result:
\begin{prop}
Let \( d \geq 3 \). Suppose \( M \subset \mathbb{R}^{d+1} \) is a hypersurface and \( x \in M \) is an equicurved point. If \( x \) is also an umbilical point, then all principal curvatures vanish at \( x \); that is,
\[
\kappa_1(x) = \kappa_2(x) = \cdots = \kappa_d(x) = 0,
\]
and hence \( M \) is flat at \( x \).
\end{prop}

It is also relatively simple to construct hypersurfaces that contain nonflat equicurved points. For example, let \( M \subset \mathbb{R}^4 \) be the graph of the smooth function
\[
f(x_1, x_2, x_3) = \frac{1}{2}(x_1^2 + x_2^2 + 4x_3^2).
\]
This defines a smooth hypersurface in \( \mathbb{R}^4 \). One can verify that the origin \( x = (0, 0, 0, 0) \) is a nonflat equicurved point of \( M \), since the principal curvatures at the origin are
\(\kappa_1 = 4,\) \( \kappa_2 = 1\), and \( \kappa_3 = 1\) which satisfy the equicurvature condition.

Currently, we are not aware of any instances of compact equicurved hypersurfaces in \( \mathbb{R}^{d+1} \) for \( d \geq 3 \). We do not delve into the construction of such examples in this paper, as it falls outside the scope of this paper. Hence, we leave this as an open question for future research.

\section*{Acknowledgements}
The authors gratefully acknowledge the generous support provided by National Cheng Kung University.

\end{document}